\documentclass[draft]{birkjour}

\usepackage[utf8]{inputenc}
\usepackage{color}
\usepackage{amscd}
\newcommand{\norm}[2]{\left\|#1\right\|_{#2}}
\newcommand{\scalar}[2]{\langle{ #1},{#2} \rangle}

\newcommand{\abs}[1]{\left\lvert #1 \right\rvert}
\newcommand{\lr}[1]{\left(#1\right)}
\newcommand{\nat}{\mathbb N}
\newcommand{\real}{\mathbb R}

\newcommand{\range}{\mathcal R}
\newcommand{\J}{\mathbf{J}_2}
\newcommand{\E}[1]{\mathbf{E}_2^{#1}}
\newtheorem{thm}{Theorem}
\newtheorem{prop}{Proposition}
\newtheorem{corollary}{Corollary}
\theoremstyle{definition}
\newtheorem{mydef}{Definition}
\newtheorem{rem}{Remark}
\newtheorem{xmpl}{Example}
\newtheorem{lem}{Lemma}
\numberwithin{equation}{section}
\newcommand{\bigo}{\mathcal O}
\newcommand{\littleo}[1]{{\scriptstyle\mathcal{O}}\lr{#1}}

\newcommand{\pmathe}[1]{#1}

\author{Peter Math\'e}
\address{Weierstra{\ss} Institute for Applied Analysis and Stochastics, Mohrenstra{\ss}e 39, 10117 Berlin,  Germany}
\email{peter.mathe@wias-berlin.de}
\author{Bernd Hofmann}
\address{Faculty of Mathematics, Chemnitz University of Technology,
 09107 Chemnitz,  Germany}
\email{hofmannb@mathematik.tu-chemnitz.de}
\title[Comparing the Ill-posedness of linear operators]{Comparing the ill-posedness for linear operators in Hilbert spaces}
\date{\today}
\dedicatory{\raggedright If a mthematical theory has not yet proved useful,\\\raggedright then it
  should at least be beautiful. \\[1ex]Dedicated to the memory of Albrecht Pietsch,
  a distinguished scientist and teacher}
\begin{document}
\begin{abstract}
The difficulty for solving ill-posed linear operator equations in
Hilbert space is reflected by the strength of ill-posedness of the governing operator,
and the inherent solution smoothness. In this study we focus on the
ill-posedness of the operator, and we propose a partial ordering for
the class of all bounded linear operators which lead to ill-posed operator
equations.
For compact linear operators, there is a simple characterization in
terms of the decay rates of the singular values. In the context of
the validity of the spectral theorem the partial ordering can also be
understood.
We highlight that range inclusions yield partial ordering, and we
discuss cases when compositions of compact and non-compact operators
occur. Several examples complement the theoretical results.
 \end{abstract}

\maketitle

\section{Introduction}
\label{sec:intro}

The goal of this study is to \emph{compare the strength (degree) of ill-posedness} between \emph{two} linear operator equations
\begin{equation} \label{eq:opeq1}
A x = y \qquad (x \in X, \,y \in Y)
\end{equation}
and
\begin{equation} \label{eq:opeq2}
A^\prime x^\prime = y^\prime \qquad (x^\prime \in X^\prime, \,y^\prime \in Y^\prime).
\end{equation}
Both equations represent mathematical models for inverse problems, and they are characterized by the \emph{injective} and \emph{bounded} linear forward operators $A: X \to Y$ and $A^\prime: X^\prime \to Y^\prime$ mapping between the
\emph{infinite dimensional and separable real Hilbert spaces} $X,Y$ as
well as $X^\prime,Y^\prime$. It is supposed that the operators $A$ and
$A^\prime$ possess non-closed ranges $\mathcal{R}(A)$ and
$\mathcal{R}(A^\prime)$.
This has the consequence that both operator equations~\eqref{eq:opeq1} and~\eqref{eq:opeq2} are ill-posed.
To avoid additional technical notation we
shall assume that the operators have dense ranges in~$Y$
and~$Y^{\prime}$, respectively.
To simplify the formulation, we speak in the following of comparing
the strength of the ill-posedness of the operators $A$ and $A^\prime$,
but actually mean the comparison of the strength of ill-posedness of
the operator equations~\eqref{eq:opeq1} and~\eqref{eq:opeq2}.

The comparison of the ill-posedness of linear operator equations was
raised earlier, especially when comparing equations with compact and
non-compact operators. In \cite[p.55]{Nashed86}, M.~Z.~Nashed states that
``\dots an equation involving a bounded non-compact operator with non-closed range is
`less' ill-posed than an equation with a compact operator with
infinite dimensional range''. Often the comparison of compact operators  is based
on the \emph{degree of ill-posedness} or the \emph{interval of ill-posedness},
and we refer to \cite{HofTau97} for formal definitions.
Such comparison seems rough, as there are very different operators, sharing
the same degree of ill-posedness. It may even happen that the degree of ill-posedness equals zero, although
the operator equation is ill-posed. Other authors consider the growth rate of distribution functions, as in \cite{WH24},
or the decay rate of decreasing rearrangement, see \cite{MNH22}.

Here we introduce a direct comparison by means of a partial
ordering. Such partial ordering must take into account that the
governing operators~$A$ and $A^{\prime}$ may have different domain and
target spaces. The formal definition is given in
Section~\ref{sec:ordering}. Then we show in
Section~\ref{sec:ordering-compact} that for compact operators
the present definition coincides with comparing the decay rates of the
singular values. In a side step, we touch the cases when the
compact operator~$A^{\prime}$ is a composition of another compact operator $A$
with a non-compact one possessing a non-closed range, see Section~\ref{sec:non}.
 Comparison of two operators may also be viewed as comparison of the
 ranges, both dense in their target spaces. We highlight in
 Section~\ref{sec:inclusions} that this is
 indeed covered by the present approach.
 Section~\ref{sec:regularization} indicates how range inclusions have
 impact on regularization theory; and this was actually the starting
 point for our investigations.
 The comparison of non-compact operators may be reduced to comparing
 the self-adjoint non-negative companions via the polar
 decomposition. In Section~\ref{sec:spectral} we use the spectral
 theorem to indicate when and how such non-compact operators can be
 compared.
 Throughout the study, we present examples that explain our
 attitude, and the final Section~\ref{sec:rangestudies} complements
 the analysis with two further examples.

\section{Partial ordering of operators: Definition and general results}
\label{sec:ordering}

For the bounded linear operators $A$ and $A^\prime$ with non-closed
ranges, we introduce a partial ordering as follows:
\begin{mydef}
  [partial ordering] \label{def:ordering}
   The operator~$ A^{\prime} \colon X^{\prime} \to Y^{\prime}$ is said to be \emph{more ill-posed than}
$A \colon X \to Y$ whenever there exist a bounded linear
 operator $S\colon X^{\prime} \to X$ and an orthogonal operator $R\colon Y \to Y^{\prime}$
 such that~$A^{\prime} = R\, A\, S$. In this case we shall
 write~$A^{\prime} \prec_{R,S} A$.
\end{mydef}
\begin{rem} \label{rem:strict}
  Not every pair of operators may be comparable in the sense of the partial ordering introduced by Definition~\ref{def:ordering}.
  If a pair $S$ and $R$ of operators with~$A^{\prime} \prec_{R,S} A$ does not exist, we shall write $A^\prime \not \prec A$.
 If we have that $A^{\prime} \prec_{R,S} A$ but~$A \not \prec A^\prime$, then the operator $A^{\prime}$ is said to be \emph{strictly more ill-posed than}~$A$.
 If, however, $A^\prime \not \prec A$ as well as $A \not \prec
 A^\prime$, then $A$ and $A^\prime$ are said to be
 \emph{incomparable}. Finally we shall write~$A^{\prime}\asymp A$ if
 both~$A^{\prime} \prec A$ and~$A \prec A^{\prime}$. With respect to Definition~\ref{def:ordering}, the relation `$\asymp$' represents the specific kind of `equality' of two operators such that the partial ordering is not only reflexive and transitive, but also antisymmetric.
\end{rem}

\begin{figure}
\center
  \centering
 \begin{equation*}
\begin{CD}
    X^{\prime} @> A^{\prime}>> Y^{\prime}\\
  @VS VV @AA R  A \\
    X  @> A >>Y
  \end{CD}
\end{equation*}
\caption{Comparison of operators $A$ and $A^\prime$ via factorization, where using the orthogonal mapping $R$.}
  \label{fig:comparison}
\end{figure}

Factorization as shown in Figure~\ref{fig:comparison} can hold true in many ways. There are two notable
situations, namely the one-sided compositions, i.e.,\ when
either $A^{\prime}= T A$, or $A^{\prime}= A S$. These are called \emph{factorization
from the left} or \emph{factorization from the right}, respectively.
Let us mention that factorization from the left with $T=R$ being
orthogonal is not interesting, because then we have that also~$A = R^{\ast}A^{\prime}$, and hence that~$A^{\prime} \asymp A$.
However, the situation for~$A^\prime=T A$, and non-orthogonal
operator~$T$ is less trivial, and we refer
to examples in Section~\ref{sec:non}.

A natural example for factorization from the right~$A^{\prime}= A S$, and hence ordering,
occurs when~$X^{\prime} \subset X$, i.e., when $X^{\prime}$
is continuously embedding into~$X$ via the mapping~$S=J_{X^{\prime}}^{X}$, and the mapping~$A^{\prime}$ is the
restriction of~$A\colon X\to Y$. In this case we can write~$A^{\prime}
= A \, J_{X^{\prime}}^{X}$, thus~$A^{\prime}
\prec_{I_{Y},J_{X^{\prime}}^{X}} A$.

Finally we mention that the partial ordering is compatible with the absolute
value. Indeed, given a bounded operator~$A\colon X \to Y$ we
assign the absolute value operator $\abs{A}:= \lr{A^{\ast}A}^{1/2} \colon X \to X$ .
By virtue of the polar decomposition, see~\cite[Chapt.~1,
3.9]{MR1721402} the following result holds true.
\begin{prop}\label{prop:absolute}
  Let~$A$ be a bounded operator and~$\abs{A}$ its absolute
  value. Then we have that~$A \asymp \abs{A}$.
\end{prop}
\begin{proof}
  By assumption we have that $A = U \abs{A}$ for an orthogonal operator
$U \colon X \to Y$, as well as~$\abs{A} = U^{\ast} A$. Therefore
$A \prec_{U, I_{X}} \abs{A}$, and~$\abs{A} \prec_{U^{\ast},I_{X}} A$,
which yields the conclusion.
\end{proof}

\section{Partial ordering of compact operators}
\label{sec:ordering-compact}

This definition, depicted in the Figure~\ref{fig:comparison}, is well-motivated for \emph{compact} operators $A$ and $A^\prime$.
As the following two propositions show, Definition~\ref{def:ordering} compares in that case the behavior of the decay rates of singular values $s_n$ to zero as $n \to \infty$ of
the operators $A$ and $A^{\prime}$.

\begin{prop}\label{prop:decay-rates1}
  Suppose that the compact operator $A^\prime \colon X^\prime \to Y^\prime$ is more ill-posed than $A \colon X \to Y$, i.e.~$A^{\prime} \prec_{R,S} A$. Then the decay rate to zero of the singular values of $A^\prime$ is not slower than the corresponding decay rate of $A$, which means that
 \begin{equation} \label{eq:notslower}
s_{n}(A^{\prime}) = \mathcal O(s_{n}(A)) \quad \mbox{as} \quad n\to\infty.
 \end{equation}
\end{prop}
\begin{proof}
Indeed, $A^{\prime} \prec_{R,S} A$  means that there are a bounded linear operator~$S\colon X^{\prime} \to X$ and an orthogonal operator $R\colon Y \to Y^{\prime}$
 such that $A^{\prime} = R\, A\, S$. Then,  by the definition of singular values, we find that
$$
s_{n}(A^{\prime}) \leq \norm{R}{Y
  \to Y^{\prime}} s_{n}(A) \norm{S}{X^{\prime}\to X}\quad (n=1,2,\dots),
$$
which implies that~\eqref{eq:notslower} is valid.
\end{proof}

The converse assertion holds also true.
\begin{prop}\label{prop:decay-rates2}
  Suppose that the compact operators~$A \colon X \to Y$ and $A^{\prime} \colon X^\prime \to Y^\prime$ are such
  that~$s_{n}(A^{\prime}) = \mathcal O(s_{n}(A))$ as~$n\to\infty$.
  Then there exist a bounded linear operator $S\colon X^{\prime} \to X$ and an orthogonal operator
  $R\colon Y \to Y^{\prime}$
 such that~$A^{\prime}  \prec _{R,S}A$.
\end{prop}
\begin{proof} By the Schmidt Representation Theorem (SVD) we have
  orthonormal systems $u_{j},v_{j}\;(j\in\nat)$
  and $u_{j}^{\prime},v_{j}^{\prime}\;(j\in\nat)$ such that~$A = \sum_{i=1}^{\infty}s_{i}(A) u_{i} \otimes v_{i}$,
  and~$A^{\prime} = \sum_{i=1}^{\infty}
  s_{i}(A^{\prime})u_{i}^{\prime} \otimes v_{i}^{\prime}$.
  By assumption the sequence $\sigma_{i}:=
  \frac{s_{i}(A^{\prime})}{s_{i}(A)},\ (i\in\nat)$, is bounded.
  Now we can set $S_{\sigma}:= \sum_{i=1}^{\infty} \sigma_{i} u_{i}^{\prime}
  \otimes u_{i}$  and~$R:=  \sum_{i=1}^{\infty} v_{i}\otimes
  v_{i}^{\prime}$. By construction, the operator~$S_{\sigma}$ is
  bounded, and $R$ constitutes an orthogonal mapping. It
  is straightforward to see that~$A^{\prime} = R \, A \, S_{\sigma}$.
\end{proof}
Both the above propositions yield
\begin{corollary}\label{cor:asymp}
  Let~$A$ and~$A^{\prime}$ be compact operators.
  The following assertions are equivalent.
  \begin{enumerate}
  \item We have that~$A^{\prime} \asymp A$.
  \item It holds true that~~$s_{n}(A^{\prime}) \asymp s_{n}(A)$ as~$n\to\infty$.
  \end{enumerate}
\end{corollary}
\begin{rem} \label{rem:littleo}
By Definition~\ref{def:ordering} the operator~$A^\prime$ is strictly
more ill-posed than $A$, if~$A^{\prime}  \prec _{R,S}A$, but there are
no corresponding operators $S^\prime$ (bounded linear) and $R^\prime$
(orthogonal) such that $A \prec _{R^\prime,S^\prime} A^\prime$. For compact operators, this is the case
under the stronger assumption
\begin{equation} \label{eq:faster}
s_{n}(A^{\prime}) = \littleo {s_{n}(A)} \quad \mbox{as} \quad n\to\infty.
 \end{equation}
\end{rem}

\section{Composition with non-compact operators} \label{sec:non}

We extend the previous discussion with the following observation.
Suppose that
we have a factorization from the left in the form $A^{\prime} = T A $ or from the right in the form $A^\prime = A\, T$
for a compact operator $A$ and a compact or a bounded non-compact linear operator $T$. Then we find that $A^\prime$ is also compact and obeys the condition
$s_{n}(A^{\prime}) =
\bigo\lr{s_{n}(A)}$ as $n \to \infty$ for the singular values of $A$ and $A^\prime$. Now, Proposition~\ref{prop:decay-rates2} implies
that~$A^{\prime}\prec A$, and there will thus be a
factorization~$A^{\prime} = R A S$ with an orthogonal operator $R$ and a bounded operator $S$.

A typical situation for the factorization from the left is met when the operators
are connected by some bounded but non-compact multiplication operator $T$ as discussed in the following example:
\begin{xmpl}
Let $X=X^\prime=Y=Y^\prime=L_2(0,1)$ and consider compositions $A^\prime =T A$ with non-compact multiplication operators $T:=M$ defined as
 $$(Mx)(t):=f(t)\,x(t)\quad(0 \le t \le 1)$$
and mapping in $L_2(0,1)$. The occurring non-negative multiplier functions $f\in L_\infty(0,1)$ are supposed to possess essential zeros. Such compositions with focus on the simple integration operator $A:=J$ defined as $$(Jx)(s):= \int_0^s x(t) dt\quad(0 \le s \le 1)$$
were considered in~~\cite{HeinHof03,HofWolf05,HofWolf09}. It was shown there that for wide classes of functions $f$, including the
monomials $f(t)=t^\kappa$ for all $\kappa>0$, the decay rates of the
singular values of $J$ and $M J$ coincide, which implies that
$J\asymp MJ$ for such multiplier functions~$f$.
\end{xmpl}

Another aspect for factorizations from the left is highlighted by the
following general assertions, Proposition~\ref{pro:Pietsch}, and its
Corollary~\ref{cor:Pietsch}, here with focus on ill-posed situations.
The subsequent Example~\ref{ex:Hausdorff} below illustrates the situation of the corollary.

First the following result, privately communicated by
A.~Pietsch, seems interesting.
\renewcommand{\theenumi}{(\alph{enumi})}
\begin{prop} \label{pro:Pietsch}
 The following two assertions are equivalent for an arbitrary bounded linear operator $T\colon Y \to Z$.
  \begin{enumerate}
  \item\label{it:ac} There is a constant~$c>0$ such that~$s_{n}(A)\leq c s_{n}(T
    A)$ for all compact linear operators $A\colon X \to Y$, and all~$n=1,2,\dots$.
  \item \label{it:left-inverse}The operator~$T$ admits a bounded left inverse $B \colon Z \to Y$, i.e.,\ $B\, T = I_{Y}$.
  \end{enumerate}
\end{prop}
\begin{proof}
Clearly, item~\ref{it:left-inverse} yields~\ref{it:ac}, because then
  $$
s_{n}(A) = s_{n}(B\, T A) \leq \norm{B}{Z \to Y} s_{n}(T A),\quad n=1,2,\dots.
$$
Next, if item~\ref{it:ac} holds true, then this must hold for
arbitrary rank one operators~$A= x_{0}\otimes y_{0}$, mapping~$x \mapsto \scalar{x}{x_{0}}y_{0},\ x\in X$, and for~$n=1$. In this case the
assumption with~$n=1$ translates to
$$
\norm{x_{0}}{X}\norm{y_{0}}{Y} = s_{1}(x_{0}\otimes y_{0})
\leq c s_{1}(x_{0}\otimes T y_{0} ) = c \norm{ T y_{0}}{Z}\norm{x_{0}}{X}
$$
Since~$x_{0}\neq 0$ and~$y_{0}\neq 0$ may be chosen arbitrary we find
that the operator $T \colon Y \to \range(T)$ is continuously invertible, and the inverse $T^{-1}$ can be continuously extended to $\overline{\range(T)}$. Denoting~$P\colon Z \to \overline{\range(T)}$ the projection, the mapping~$B:= T^{-1}P$
constitutes the left inverse to $T$. The proof is complete.
\end{proof}
Within the present context of ill-posed operator equations this yields the following.
\begin{corollary} \label{cor:Pietsch}
Let~$T \colon Y \to Z$ be an injective bounded linear operator with
  non-closed and dense range. Then for every (arbitrarily large)
  constant $C<\infty$ there are a compact operator $A \colon X \to Y$ and an index~$n$
  such that $s_{n}(A)/s_{n}(T A) \geq C $.
\end{corollary}
\begin{proof}
Since we assume that~$\overline{\range(T)}=Z$, the existence of a
bounded left inverse $B \colon Z \to Y$ to $T$ actually  requires the
existence of a  bounded
inverse~$T^{-1}$. This contradicts the ill-posedness of $T$ coming from the non-closedness of the range $\range(T)$.
Hence,  item~\ref{it:ac} of Proposition~\ref{pro:Pietsch} is violated. The
assertion of the corollary is a reformulation of the violation of item~\ref{it:ac}.
\end{proof}

\begin{xmpl} \label{ex:Hausdorff}
Here we consider the non-compact ill-posed Hausdorff moment operator $T:=B^{(H)}$.
This operator $B^{(H)} \colon Y=L_2(0,1) \to Z=\ell^2$ is given as
  \begin{equation} \label{eq:Haus}
[B^{(H)}z]_j:= \int_0^1 t^{j-1}z(t)dt \qquad (j=1,2,\dots).
\end{equation}
In the composition $A^{\prime} := B^{H}\circ J$ with the simple integration operator $A:=J$ acting in
$X=Y=L_2(0,1)$ the situation of Corollary~\ref{cor:Pietsch} is met.
Indeed, it was shown in \cite{HofMat22} that we actually have that $s_{n}(A^{\prime}) =
\littleo{s_{n}(J)}$ as $n\to\infty$, and hence $B^{H}\circ J$
is strictly more ill-posed than $J$.

Also for the composition of the Hausdorff moment operator  $T:=B^{(H)}$ with the compact embedding operator $A:=E^k \colon H^k(0,1) \to
  L_2(0,1)$, mapping from the Hilbertian Sobolev spaces $X=H^k(0,1)\;(k=1,2,...)$ to
  $Y=L_2(0,1)$, the situation of Corollary~\ref{cor:Pietsch} occurs.
In both cases, $B^{(H)}J$ and $B^{(H)} E^k$, it is still an open problem
whether the composition operator has power type or exponential
decay of the singular values.

Subsequently, a series of studies turned to such questions, see e.g.~\cite{DFH24}, and references therein.
\end{xmpl}

As we have seen, Definition~\ref{def:ordering} also applies in combination with non-compact bounded linear operators. A direct comparison is in the focus of the following proposition. It will highlight the fact that bounded non-compact linear operators with non-closed ranges can never be more ill-posed in the sense of Definition~\ref{def:ordering} than compact linear operators.

\begin{prop} \label{pro:comnoncom}
  Suppose that~$A$ is a \emph{non-compact} bounded linear operator with non-closed range,
  and that~$A^{\prime}$ a \emph{compact} linear operator with infinite dimensional range.
  Then we  have that~$A \not \prec A^\prime$. If, however, $A^{\prime} \prec A$ holds, then $A^\prime$ is even \emph{strictly more ill-posed} than $A$.
\end{prop}
\begin{proof}
Suppose to the contrary that~$A \prec A^\prime$ holds true, and hence
a factorization~$A = R^\prime A^\prime S^\prime$ exists. The family of
compact operators constitutes an operator ideal, and hence the
composition~$R^\prime A^\prime S^\prime$ will be a compact operator,
which contradicts the assumption. The second assertion holds
  true because under the made assumptions we cannot have that~$A \asymp A^{\prime}$.
\end{proof}

The partial ordering of Definition~\ref{def:ordering}
  characterizes compact linear operators $A^\prime$ as strictly more ill-posed
  than non-compact bounded linear operators $A$ whenever $A^\prime$ and $A$ are comparable, see in this context also Remark~\ref{rem:ri}. However, a comparison with respect to the strength of ill-posedness of compact and ill-posed non-compact operators has
  various facets and was discussed contradictorily in the
  literature. As mentioned in the introduction,  Nashed's  opposite claim
in \cite[p.55]{Nashed86} is rather problematic due to occurring incomparability phenomena.
 We refer in this context also to the study~\cite{Hofkind10}, where it has been shown that compact and non-compact operators are fundamentally different with respect to projections into finite dimensional spaces. The reason for this seems to be that sequences of compact operators can never converge in norm to a non-compact operator. On the other hand, we note that a comparison with respect to the strength ill-posedness of two compact operators $A$ and $A^\prime$ with each other is consistent along the lines of Section~\ref{sec:ordering-compact} and Definition~\ref{def:ordering}. Also the comparison of two ill-posed bounded non-compact operators $A$ and $A^\prime$ with each other is possible by Halmos' spectral theorem when unitary transformations lead to a common measure space for both operators. This will be outlined in detail below in Section~\ref{sec:spectral}. Alternatively, the comparison of two ill-posed non-compact operators can also be done in a consistent manner by comparing the growth rate of distribution functions (see~\cite{WH24}) or by comparing the decay rate of decreasing rearrangements (see~\cite{MNH22}).
 Both alternative approaches exploit the spectral theorem in multiplication operator form as well.

\section{Range inclusions yield ordering}
\label{sec:inclusions}

A further strong motivation for Definition~\ref{def:ordering} is due to range inclusions as a tool for comparing the ill-posedness of two operators $A$ and $A^\prime$.
It can be expected that \emph{smaller ranges of operators} $A^\prime$ compared to $A$ \emph{indicate a higher degree of ill-posedness}.
A stringent justification of this  comes from Douglas'
\pmathe{Range Inclusion Theorem~\cite{MR0203464}. Here we sketch the
  main facts with slightly different notation.}
\begin{thm}\label{thm:Douglas-original}
  Let~$H$ and~$G$ be bounded linear operators in some Hilbert
  space~$\mathcal H$. The
  following assertions are equivalent.
  \begin{enumerate}
  \item $\mathcal R(H) \subset \mathcal R(G)$;
  \item For some~$\lambda \geq 0$ we have that~$H H^{\ast} \leq G
    G^{\ast}$;
  \item There exists a bounded operator~$C$ with~$\norm{C}{}\leq
    \lambda$ on~$\mathcal H$ such that~$H = G C$.
  \end{enumerate}
\end{thm}
In a note added to the proof the author mentions that the results
extends to operators acting from different Hilbert spaces into the
common range space~$\mathcal H$.

In light of this theorem, and using the recent
formulation in~\cite{MR3985479}, we can formulate the following
result.
\begin{thm}\label{thm:Douglas}
  The following assertions are equivalent:
\begin{enumerate}
\item\label{it:a} There exists an orthogonal
 mapping~$R\colon Y\to Y^{\prime}$, for which the range inclusion
 $$
R^{\ast} \;\mathcal R(A^{\prime})=\mathcal R (R^{\ast}A^{\prime})  \subset \mathcal R(A)
$$
is satisfied.
\item\label{it:b} There is a constant $0 \le C<\infty$ such that, for all $y \in Y$,
  $$\|\lr{A^{\prime}}^{\ast} R y \|\leq C \|A^{\ast}y \|.$$
\item\label{it:c} There exist a bounded linear operator $S\colon X^{\prime}\to X$ and an orthogonal operator $R\colon Y \to Y^{\prime}$ such that
$A^\prime$ obeys the factorization~$A^{\prime} = R \, A\, S $, which means $A^{\prime} \prec_{R,S} A$ in the sense of Definition~\ref{def:ordering}.
\end{enumerate}
\end{thm}
\begin{proof}
\pmathe{We apply the original theorem to the operators~$H:= R^{\ast}A^{\prime}$,
and~$G:= A$, both sharing the same range space~$Y$, respectively.} Since~$R$ was orthogonal, the adjoint mapping
is~$\lr{A^{\prime}}^{\ast} R$, which shows the equivalence of
item~\ref{it:a} and~\ref{it:b}. \pmathe{The norm bound is a reformulation of the
  corresponding item in Theorem~\ref{thm:Douglas-original}.}
Again, from the original
formulation we find a factor, say~$S$, such that~$R^{\ast}A^{\prime} =
A S$. Using the orthogonality once more this yields~$A^{\prime} = R A
S$, which is the assertion from item~\ref{it:c}, and the proof is
completed.

\end{proof}

The following corollary characterizes the special case, when the image
spaces of $A$ and $A^\prime$ \pmathe{are nested}, i.e.~when~\pmathe{$Y^\prime := \overline{\mathcal
R(A^{\prime})}\subset Y$}.
\renewcommand{\theenumi}{(\roman{enumi})}
\begin{corollary} \label{cor:comnoncom}
Let $A \colon X \to Y$ and $A^\prime \colon X^\prime \to Y^{\prime}
\subset Y$ be both
bounded linear operators with non-closed ranges.
The following assertions are equivalent.
\begin{enumerate}
\item There is a range inclusion~$\mathcal R (A^{\prime}) \subset
  \mathcal R (A)$.
\item There is a bounded  linear operator $S:X \to X^\prime$ such that the factorization $A^\prime= A\,S$ holds true.
\end{enumerate}
In either case we have that~$A^\prime \prec_{I_{Y},S} A$.
\end{corollary}
In particular we see that factorizations from the right always yield comparison.

\begin{rem} \label{rem:ri}
Proposition~\ref{pro:comnoncom} yields that  the range inclusion~$R (A^{\prime}) \subset \mathcal R(A)$ cannot hold when $A^\prime$ is non-compact, but $A$ is compact. The interplay of compact and non-compact operators possessing $\mathcal R(A^\prime) \subset \mathcal R(A)$ has been discussed in \cite[Example~10.2]{BHTY06},
where also a concrete example of a compact operator $A^\prime$ is
presented that is strictly more ill-posed than a concrete non-compact
operator $A$. Such situations of comparable compact and non-compact operators are typical for factorizations from the right when the compact operator $A^\prime$ is factorized as $A^\prime=A \circ T$ with an ill-posed non-compact bounded operator $A$. Here, $T$ is mostly a compact operator
(e.g. $A^\prime = B^{(H)} \circ J$ in Example~\ref{ex:Hausdorff}), but $T$ can also be a non-compact operator as comprehensively discussed in \cite{KinHof24}.

If the ranges coincide, i.e.~$R (A^{\prime}) = \mathcal
R(A)$, then by virtue of Theorem~\ref{thm:Douglas} we have
that~$A^{\prime} \asymp A$, and in particular for compact operators that $s_n(A^\prime) \asymp s_n(A)$
for $n \to \infty$, by Corollary~\ref{cor:asymp}. Note that the
pre-image spaces $X$ and $X^\prime$ in this context can be very
different.

Unfortunately, if $R (A^{\prime})$ is a \emph{proper subset} of $\mathcal R(A)$, then one cannot conclude that $A^\prime$ is strictly more ill-posed
than $A$, and we refer to Lemma~\ref{lem:codim} below for counterexamples.
\end{rem}

\begin{lem} \label{lem:codim}
Let for the compact operators with infinite dimensional range $A
\colon X \to Y$ and $A^\prime \colon X^\prime \to \pmathe{Y^{\prime}:=\overline{\mathcal
R(A^\prime)}\subset Y}$ hold that the range $Y^{\prime}$ is a $m$-codimensional subspace of $Y$ with $m \in \nat$. Moreover assume that
the singular values of $A$ satisfy the condition
\begin{equation} \label{eq:s2}
 s_{2n}(A)/s_n(A) \ge \underline C \quad \mbox{for all}\; n \in \nat \;\; \mbox{and some constant} \; \underline C>0.
\end{equation}
  Then the decay rates to zero of the singular values of $A$ and
  $A^\prime$ coincide, i.e.~$s_n(A^\prime) \asymp s_n(A)$ as $n \to
  \infty$, and hence~$A^{\prime}\asymp A$.
\end{lem}
\begin{proof}
Let $Q$ denote the orthogonal projection  \pmathe{from $Y$ onto the
$m$-codim\-ensional subspace $Y^{\prime}\subset Y$}.
We see that $A = (I - Q) A + Q \,A$, and the calculus with singular values provides us with the estimate
$$
s_{m+ n}(A) \leq s_{n}(Q\,A) + s_{m+1}((I - Q) A) =
s_{n}(Q \,A),
$$
because~$\operatorname{rk}(I- Q) = m < m+1$. Thus, we have
$$s_{2n}(A) \leq s_{n}(Q \,A) \le \hat C \,s_n(A^\prime)$$ for some constant $0<\hat C<\infty$ and $n \ge m$. The right inequality is a consequence of the construction $\mathcal R(Q\,A)=\mathcal R(A^\prime)$, which implies $\mathcal R(Q\,A) \subset \mathcal R(A^\prime)$ and yields with Corollary~\ref{cor:comnoncom} in combination with Proposition~\ref{prop:decay-rates1} the existence of the constant $\hat C$ satisfying the inequality $s_{n}(Q \,A) \le \hat C \,s_n(A^\prime)$. Moreover, there is a constant $0<\overline C<\infty$ from $\mathcal R (A^\prime) \subset \mathcal R(A)$
such that we have $s_n(A^\prime) \le \overline C\,s_n(A)$. Together we find, by recalling the condition \eqref{eq:s2} with the constant $\underline C>0$, the estimate
$$\underline C\,s_n(A) \le s_{2n}(A) \leq \hat C\,s_n(A^\prime)\le \hat C\,\overline C\,s_{n}(A) \quad (n \ge m),$$
which proves the lemma.
\end{proof}

This lemma will be applied in Section~\ref{sec:rangestudies} for comparing the ill-posedness of pairs of compact operators, where the range of one is a proper subset of the range of the other. We note that the required assumption~\eqref{eq:s2} of the lemma is satisfied if the decay rate of the singular values of $A$ is polynomial, i.e.~if there are constants $0<c_1,c_2<\infty$ and exponents $0<\theta_1 \le \theta_2<\infty$ such that $c_1\,n^{\theta_1} \le s_n(A) \le c_2\,n^{\theta_2}\;(n \in \nat)$.

\section{Impact on regularization}
\label{sec:regularization}

The comparison of the smoothing properties of the operators~$A$
and~$A^{\prime}$, say~$A^{\prime}\prec A$, may loosely be understood in the sense that the
application of the inverse of the injective mapping~$A$ may be compensated by the
application of the operator $A^{\prime}$. This reasoning would assume
that we have $\range(A^{\prime}) \subset \mathcal D(A^{-1})$, because
then~$A^{-1}\circ A^{\prime}$ could be defined. Therefore, the
orthogonal mapping~$R^{\ast}$ may be used to put the range
of~$A^{\prime}$ into the target space~$Y$ of~$A$. Then it makes sense
to consider the potentially unbounded mapping~$A^{-1} R^{\ast}
A^{\prime}$.

The solution theory of ill-posed operator equations deals with replacing the unbounded
mapping~$A^{-1}$ by bounded mappings $$g_{\alpha}(A^{\ast}A)A^{\ast}\colon Y \to X\qquad (\alpha>0),$$
which are constructed by a family of generator functions $g_\alpha$ for regularization.  We
will not dwell into details here, rather we refer to the
monograph~\cite{EHN96}. In this focus we can formulate the following result.

\begin{prop} \label{prop:regu}
  Let $g_{\alpha} $ be a family of generator functions for regularization.
  We have the following dichotomy.
  \begin{enumerate}
  \item Either $\range\lr{R^{\ast}A^{\prime}} \subset \range(A)$, and
    then
  $$
\norm{g_{\alpha}(A^{\ast}A) A^{\ast}\; R^{\ast}\;A^{\prime}}{X^{\prime}\to
  X}\ \text{is uniformly bounded as}\  \alpha\to 0,
$$
\item or $\range\lr{R^{\ast}A^{\prime}} \not \subset \range(A)$, and
  then the family
  $$\norm{g_{\alpha}(A^{\ast}A) A^{\ast}\; R^{\ast}\;A^{\prime}}{X^{\prime}\to
    X},\ \alpha>0
  $$is unbounded.
  \end{enumerate}
\end{prop}
\begin{proof}
This  is a consequence of the well-known dichotomy, see
e.g.~\cite[Chapt. 2, Thm. 5.2]{MR859375}. In ibid., this is formulated
for iterative regularization schemes. However, such philosophy is also valid for
arbitrary types of regularization. For an operator~$A$ with dense
range, it says the following:  Either~$y\in \range(A)$ and
then~$g_{\alpha}(A^{\ast}A)A^{\ast}y$ converges to~$A^{-1}y$, or~$y\not
\in \range(A)$ in which case~$\norm{g_{\alpha}(A^{\ast}A)A^{\ast}y}{X}$
is unbounded as~$\alpha\searrow 0$. The assertion of the proposition is
a direct consequence of this.
\end{proof}

\section{Application of spectral theorem to partial ordering}
\label{sec:spectral}

Definition~\ref{def:ordering} is also applicable for comparing two non-compact operators via Halmos' version (cf.~\cite{h63}) of the spectral theorem, applicable to self-adjoint linear operators. Proposition~\ref{prop:decay-rates1}, for example, which has been formulated for compact operators via decay rates
of singular values, can be extended by exploiting a measure space $(\Omega,\mu)$ to  multiplication operators $M_{f}\colon L_{2}(\Omega,\mu)\to
  L_{2}(\Omega,\mu)$ defined as
$$[M_f \xi](\omega):=f(\omega)\,\xi(\omega) \quad (\mu-a.e.\;\mbox{on}\;\Omega), $$
with multiplier functions $f \in L^\infty(\omega,\mu)$.

   We know from Proposition~\ref{prop:absolute} that it is enough to
   consider non-negative self-adjoint operators. Hence we
   consider a pair~$H\colon X \to X$ and $H^\prime\colon X^\prime \to
   X^\prime$ of non-negative self-adjoint semi-definite operators. In
   this case the occurring multiplier functions~$f$ and~$f^{\prime}$ are real-valued and non-negative.
\begin{prop}
  \label{prop:normal-comparison}
  Suppose that the two self-adjoint positive semi-definite bounded linear operators $H \colon X \to X$ and $H^\prime \colon X^\prime \to X^\prime$, both with non-closed ranges $\mathcal R(H)$ and $\mathcal R(H^\prime)$, admit a spectral
  representation with respect to a  \emph{common semi-finite measure
    space} $(\Omega,\mu)$. Thus we have representations
  $$
  H= U M_{f}U^*\quad \text{and}\quad H^{\prime} = U^{\prime} M_{f^{\prime}} U^{{\prime}*}
  $$
  with orthogonal operators $U \colon L_{2}(\Omega,\mu) \to X$,  $U^\prime \colon L_{2}(\Omega,\mu) \to X^\prime$, and with
  the two multiplication operators
  $$
  M_{f}\colon L_{2}(\Omega,\mu) \to
  L_{2}(\Omega,\mu)\quad \text{and}\quad M_{f^{\prime}}\colon L_{2}(\Omega,\mu)\to
  L_{2}(\Omega,\mu),
  $$
  respectively, characterized by the multiplier functions $f$ and $f^\prime$, which are both non-negative a.e.~on $\Omega$ and belong to $L^\infty(\Omega,\mu)$. If also the quotient function
  $\frac{f^{\prime}}f$ belongs to $L^\infty(\Omega,\mu)$, and hence the
  multiplication operator~$M_{\frac{f^{\prime}}{f}}$ is bounded,
  then $H^\prime$ is \emph{more ill-posed than} $H$, i.e.~$H^{\prime}\prec_{R,S} H$, with the orthogonal mapping $R:= U^{\prime}U^{\ast} \colon X \to X^\prime$ and the bounded linear operator $S:= U M_{f^{\prime}/f}  U^{\prime *} \colon X^\prime \to X$.
\end{prop}
\begin{proof}
  It suffices to draw the following diagram, given in Figure~\ref{fig:normal}.
\end{proof}

  \begin{figure}[ht]
    \centering
     \begin{equation*}
\begin{CD}
    X^{\prime} @> H^{\prime}>> X^{\prime}\\
  @V {U^{\prime *}} VV @AA U^{\prime}  A \\
  L_{2}(\Omega,\mu)  @> M_{f^{\prime}} >>L_{2}(\Omega,\mu) \\
  @V {M_{{f^{\prime}}/{f}}}  VV @AA I_{ L_{2}(\Omega,\mu) } A \\
  L_{2}(\Omega,\mu)  @> M_{f} >>L_{2}(\Omega,\mu) \\
  @V {U} VV @AA U^{\ast} A \\
    X @> H>> X\\
  \end{CD}
\end{equation*}
    \caption{Comparison of self-adjoint operators $H$ and $H^\prime$}
    \label{fig:normal}
  \end{figure}

\begin{rem}
  For the factorization $H= U M_{f}U^* \colon X \to X$ of the self-adjoint positive semi-definite operator $H$ with non-closed range $\mathcal R(H)$ we have that the spectrum of $H$ and the essential range of the multiplier function $f$ coincide (cf.~\cite[Theorem~2.1(g)]{Haase18}). They represent the same closed subset of the bounded interval $[0,\|H\|_{X \to X}] \subset \real$, and zero belongs to that subset, because of the ill-posedness. It is important to distinguish the case of a finite measure space $(\Omega,\mu)$ with $\mu(\Omega)<\infty$,
  where \emph{increasing rearrangements} (see, e.g.,~\cite{EHZ93}) of the multiplier function $f$ characterize the situation, and the case of an infinite measure $\mu(\Omega)=\infty$, where \emph{decreasing rearrangements} of $f$ play this role. We refer for detailed studies in this context to the recent papers \cite{MNH22} and \cite{WH24}.
\end{rem}

\begin{xmpl}\label{xmpl:multiplication1-ops}
  As a typical example for comparing non-compact operators in the case $\Omega=[0,1]$ with Lebesgue measure $\mu$ on $\real$ and $\mu(\Omega)<\infty$ serves the family of pure multiplication operators in $X=X^\prime=L_2(0,1)$ with $U=U^\prime=I: X \to X$. We compare the two bounded, non-compact, self-adjoint and positive semi-definite operators  $Hx = M_{f} x$ and $H^{\prime}x = M_{f^{\prime}}x$ mapping on $L_{2}(0,1)$ with continuous and strictly increasing multiplier functions $f(t)$ and $f^\prime(t)$ for $0 < t \le 1$ satisfying the conditions $\lim_{t \to +0}f(t)=\lim_{t \to +0}f^\prime(t)=0$.
  If there is a finite constant $C>0$ such that $\frac{f^\prime(t)}{f(t)} \le C$ for all $0<t \le 1$, then $H^\prime$ is \emph{more ill-posed than} $H$, where precisely $H^{\prime}\prec_{I,M_{f^{\prime}/f}} H$. This case occurs when $f(t)=c_1t$ and $f^\prime(t)=c_2t \;(0 < t \le 1)$ with positive constants $c_1$ and $c_2$. But then also
  $H  \prec_{I,M_{f/f^\prime}} H^\prime$ takes place. If, however, the decay rate $f^{\prime}(t)\searrow 0$ is higher than the rate of $f(t) \searrow 0$ as~$t\to +0$, we have $H^{\prime} \prec_{I,M_{f^{\prime}/f}} H$ but $H \not \prec H^{\prime}$ and $H^\prime$ is even \emph{strictly more ill-posed than} $H$ like in the example $f(t)=t^{\kappa_1}$ and
  $f^\prime(t)=\exp\left(-\frac{1}{t^{\kappa_2}}\right) \;(0 < t \le 1)$ with positive constants $\kappa_1$ and $\kappa_2$.
\end{xmpl}

\begin{xmpl}\label{xmpl:multiplication2-ops}
  An example for non-compact operators analog to Example~\ref{xmpl:multiplication1-ops}, but in the infinite measure case $\Omega=[0,\infty)$ with Lebesgue measure $\mu$ on $\real$ and pure multiplication operators in $X=X^\prime=L_2(0,\infty)$ can be found by comparing continuous multiplier functions $f(t)$ and $f^\prime(t)$, which are strictly decreasing on $[0,\infty)$ and tend to zero as $t \to \infty$. Here, the decay rate of the multiplier functions at infinity determines the partial ordering. If the decay rate of $f^\prime$ at infinity is higher than those of $f$,
  then $H^\prime$ is strictly more ill-posed than $H$, i.e.~$H^\prime \prec_{I,M_{f^{\prime}/f}} H$ but $H \not \prec H^\prime$. This if for example the case for $f(t)=(1+t)^{-1}$ and $f^\prime(t)=(1+t)^{-2}\;(0 \le t < \infty)$.
  \end{xmpl}

\section{Further examples}
\label{sec:rangestudies}

In this section, we are going to study by means of examples the partial ordering of compact operators with common image spaces $Y=Y^\prime$. This case has already been
discussed in Corollary~\ref{cor:comnoncom}.
The goal is to check whether the previous results allow us to draw
conclusions about the comparability for specific operators.
\begin{xmpl} \label{xmpl:JE}
  First let us compare
  \begin{enumerate}
  \item
  the Riemann–Liouville fractional integration
operators
$J^m: L_2(0,1) \to L_2(0,1)$ of order $m=1,2,...$, defined as
$$
[J^m x](s):= \int_0^{s} \frac{(s-t)^{m-1}}{\Gamma(m)} x(t)\, d t \quad
(0 \le s \le 1),
$$
for which solving the associated operator equation \eqref{eq:opeq1}
requires to find the $m$-th derivative of the function $y$, and
\item the embedding operators $E^k \colon H^k(0,1) \to L_2(0,1)$ mapping from the Hilbertian Sobolev spaces $H^k$ to $L_2$ for $k=1,2,...$.
\end{enumerate}
Of course, the decay of the singular values of the embeddings are
known, specifically we have that
\begin{equation} \label{eq:Ek}
s_n(E^k) \asymp n^{-k}  \quad \mbox{as} \quad n \to \infty
\end{equation}
(see, e.g.,~\cite{Koenig86}).

Also, the ranges~$\range(J^{m})$ of the mapping~$J^{m}$ can be
characterized. These are known to be subsets of~$H^{m}(0,1)$, with a
finite set of linear constraints, thus the range is a
finite-codimensional subspace of~$H^{m}(0,1)$. From this we conclude
by virtue of Corollary~\ref{cor:comnoncom} that~$J^{m} \prec
E^{m}$. Moreover, since the decay rate given in~(\ref{eq:Ek}) is
polynomial, Lemma~\ref{lem:codim} applies and actually yields
that~$J^{m} \asymp E^{m}$. Consequently, by Corollary~\ref{cor:asymp}
we find that~$s_{n}(J^{m})\asymp n^{-m}$ as~$n\to\infty$. Of course,
this is known, and we refer to~\cite{VuGo94}, for instance.
Next, having these decay rates for both operators~$E^{m}$ and~$J^{m}$
we can apply Proposition~\ref{prop:decay-rates2}, see Remark~\ref{rem:littleo}, to immediately
conclude that for~$k>m$ the operator~$E^k$ is strictly more ill-posed
than $J^m$, whereas for~$m>k$ the opposite assertion holds true.
\end{xmpl}

The following example deals with functions of two real variables defined on the unit square $[0,1]^2$.
\begin{xmpl} \label{xmpl:HF}
  Here we compare for $m=1,2,...$
  \begin{enumerate}
  \item
  the compact embedding operators $\E m: H^m([0,1]^2) \to
  L_2([0,1]^2)$, with
  \item  the compact \emph{mixed integration operator} $\J: L_2([0,1]^2) \to L_2([0,1]^2)$ defined as
    $$
    [\J\, x(t_1,t_2)](s_1,s_2):= \int_0^{s_1} \int_0^{s_2}
    x(t_1,t_2)\, d t_1 \,d t_2 \quad ((s_1,s_2) \in [0,1]^2).
    $$
The associated operator equation \eqref{eq:opeq1} aims at finding the second mixed derivative $\frac{\partial}{\partial s_1 \partial s_2} y(s_1,s_2)$ of the right-hand side function $y$ of \eqref{eq:opeq1}.
\end{enumerate}
We first discuss smoothness~$m=1$. In this case we know
that~$\range(\J) \subset H^{1}(0,1)^{2} = \range(\E 1)$, and
hence~$\J \prec \E m$.

However, for~$m=2$ we have that the range $\mathcal
R(\E 2)=H^2([0,1]^2)$,  but we have $\mathcal R(\E 2) \not
\subset \range(\J)$ as outlined  in \cite[\S 4]{HF23} by a
counterexample. Thus range inclusions do not apply here. However, see
e.g., \cite[\S 3c]{Koenig86},   we
know that
\begin{equation} \label{eq:E2m}
s_n(\E 2) \asymp n^{-1}   \quad \mbox{as} \quad n \to
\infty.
\end{equation}
Thus
$$
s_n(\E 2) = \littleo {\frac{\log(n)}n}\;\; \text{as}\quad n\to\infty,
$$
where the right hand side rate corresponds to the decay rate of the
singular values, i.e.,\ $s_{n}(\J)\asymp \frac{\log(n)}n$, see, e.g.,
\cite[Prop.~3.1]{HF23}. By virtue of
Proposition~\ref{prop:decay-rates2} there must hold true
that~$\E 2\prec_{R,S} \J$, for some orthogonal mapping~$R$
and bounded~$S$. This factorization is not known
to us.
\end{xmpl}
\section*{Discussion}
\label{sec:discussion}
Two referees raised several issues related to the present study. We
acknowledge these,  and we comment on some of them, here.
\par
Firstly, does a comparison~$A^{\prime} \prec A$ imply that
also~$\lr{A^{\prime}}^{\ast} \prec A^{\ast}$? This does not follow
from the definition, because the orthogonal mapping~$R$ appears only
on the range side. However, for compact operators we may use
Proposition~\ref{prop:absolute}, because in this case we need to
compare~$|A|^{2}= A^{\ast} A$ with~$| A^{ \ast}|^{2} = A A^{\ast}$,
and these are related operators that share the same sequence of
singular values, and hence Propositions~\ref{prop:decay-rates1}
and~\ref{prop:decay-rates2} apply, and answer the question in the
affirmative for compact operators.
\par
Ill-posed equations are studied for linear as well as non-linear
mappings. Then the question is, whether non-linear Frèchet
differentiable mappings can be 
compared, for instance by using their Frèchet
derivatives. Of course, the Frèchet derivatives vary from point to
point, and hence comparison must be local.
Accordingly, an attempt to define a local degree of ill-posedness for nonlinear problems was made in~\cite{Hof94}.
In~\cite{Hof98}, one even finds the introduction of a partial ordering
in this context, and  the degree of ill-posedness, and usual partial orderings can be exploited.
The Fr\'echet derivative seems to be the optimal local linearization in this context. However, as already mentioned in~\cite{HofSch94}, there are nonlinear operators, for which the Fr\'echet derivative does not characterize the nonlinear operator locally sufficiently well. This occurs in general
when no tangential cone condition applies, for example in case of the
non-compact autoconvolution operator mapping in $L^2(0,1)$ with
compact Fr\'echet derivative everywhere. 
\par
An alternative approach to the degree of ill-posedness based on the
spectral theorem in multiplication operator form was recently
presented in~\cite{WH24}. In contrast to the present work, where the
comparison of pairs of linear operators with respect to the
ill-posedness is investigated, the focus of that study is on the characterization of whole classes of mildly, moderately and severely ill-posed problems by means of growth rates of distribution functions and decay rates of associated decreasing rearrangements. 


\end{document}